\documentclass[a4paper,12pt]{article}
\usepackage{amsmath,amsfonts,enumerate,amssymb,amsthm,color}
\title{On groups all of whose undirected Cayley graphs of bounded valency are integral}
\author{Istv\'an Est\'elyi,$^{a,b}$   \; Istv\'an Kov\'acs$^{\; b}$ \\  [+0.75ex]
$^a$ {\small FMF, University of Ljubljana, Jadranska 19, 1000 Ljubljana, Slovenia} \\ [-0.5ex]
$^b$ {\small IAM, University of Primorska, Muzejski trg 2, 6000 Koper, Slovenia} 
}
\date{}

\newtheorem{thm}{Theorem}[section]
\newtheorem{lem}[thm]{Lemma}
\newtheorem{cor}[thm]{Corollary}
\newtheorem{prop}[thm]{Proposition}

\def\G{\mathcal{G}}
\def\R{\mathcal{R}}
\def\bR{\mathbf{R}}
\def\S{\mathbf{S}}
\def\Z{\mathbb{Z}}
\DeclareMathOperator{\cay}{Cay}

\newcommand{\comment}[1]{}

\begin{document}

\maketitle

\let\thefootnote\relax\footnote{
The first author was supported by the Young Researcher grant of ARRS (Agencija za raziskovanje Republike Slovenija),
and the ARRS grant no.\ P1-0294. The second author was supported  by the ARRS grant no.\ P1-0285. 
\\  [+0.5ex]
{\it  E-mail addresses:} istvan.estelyi@student.fmf.uni-lj.si (Istv\'an Est\'elyi),  istvan.kovacs@upr.si (Istv\'an Kov\'acs).
}

\begin{abstract}
A finite group $G$ is called Cayley integral if all undirected Cayley graphs over $G$ are integral, 
i.e., all eigenvalues of the graphs are integers. The Cayley integral groups have been determined by 
Kloster and Sander in the abelian case, and by Abdollahi and Jazaeri, and independently by 
Ahmady, Bell and Mohar in the non-abelian case. In this paper we generalize this class of groups by 
introducing the class $\G_k$ of finite groups $G$ for which all graphs $\cay(G,S)$ are integral if $|S| \le k$.  
It will be proved that $\G_k$ consists of the Cayley integral groups if $k \ge 6;$  and 
the classes $\G_4$ and $\G_5$ are equal,  and consist of:\ (1) the Cayley integral groups, (2) the 
generalized dicyclic groups $Dic(E_{3^n} \times \Z_6),$ where $n \ge 1$. \medskip

 \medskip

\end{abstract}

\section{Introduction}

A finite simple graph is called \emph{integral} if all its eigenvalues are integers. 
These graphs were introduced by Harary and Schwenk \cite{HarS74}, and have been attracting 
considerable attention (see the surveys \cite{AhmABS09,BalCRSS03}). 
In \cite{AbdV09}, Abdollahi and Vatandoost proposed to consider those integral graphs which are 
also Cayley graphs. Recall that, given a finite group $G$ and a subset $S$ of $G$ with $1 \notin S$ and 
$S=S^{-1}=\{s^{-1} : s \in S\},$ the \emph{Cayley graph} $\cay(G,S)$ has vertex set $G,$ and edges in the form 
$\{g,s g\},$ $g \in G$ and $s \in S$. Klotz and  Sander \cite{KloS10} called a finite group $G$ \emph{Cayley integral} if  
all graphs $\cay(G,S)$ are integral; furthermore, they determined the abelian Cayley integral groups (see \cite[Theorem 13]{KloS10}):

\begin{thm}\label{KS}{\rm  (Klotz and Sander \cite{KloS10})}  
The only finite abelian Cayley integral groups are $E_{2^n} \times E_{3^m}$ and $E_{2^n} \times \Z_4^m,$ where $m,n \ge 0$.   
\end{thm}

The non-abelian Cayley integral groups were found recently by Abdollahi and Jazaeri (see \cite[Theorem 1.1]{AbdJ14}), and independently by Ahmady et al. (see \cite[Theorem 4.2]{AhmBM13}): 

\begin{thm}\label{ABM}{\rm  (Abdollahi and Jazaeri \cite{AbdJ14};  Ahmadi et al.\ \cite{AhmBM13})}
The only finite non-abelian Cayley integral groups are 
$D_6, \, Dic_{12}$ and $Q_8 \times E_{2^n},$ where $n \ge 0$.  
\end{thm}

The cyclic, the elementary abelian, and   
the dihedral group of order $n$ are denoted by $\Z_n, \, E_n$ and $D_n,$ respectively. 
Let $A$ be an abelian group having a unique involution $t$ and of order $|A| > 2$. 
The {\em generalized dicyclic group} $Dic(A) = \langle A, x \rangle,$  where 
$x^2=t,$ and $a^x=a^{-1}$ for every $a \in A$ (see \cite[page 252]{Sco64}).
 In the case when $A \cong \Z_n$ it is also called the {\em dicyclic group} of order $2n,$ denoted by 
$Dic_{2n};$ and when $A \cong \Z_{2^n}$ it is also  known as the {\em generalized quaternion group} of order 
$2^{n+1},$ denoted by $Q_{2^{n+1}}$.  \medskip

In this paper we are going to study groups $G$ for which we require $\cay(G,S)$ to be integral only 
when $|S|$ is bounded by a fixed number.  Formally, for $k \in \mathbb{N},$ we set  
$$
\G_k = \big\{  G \,  : \, \cay(G,S) \text{ is integral whenever } |S| \le k \big\}.
$$

It is obvious that $\G_1$ is just the class of all finite groups, and in $\G_2$ there are exactly the groups  
whose non-identity elements are of order $2,3,4$ or $6$. 

The class $\G_3$ is the most intricate. Regarding $p$-groups, 
it is clear that all groups of exponent  $3$ are in $\G_3,$ and we will show that 
a non-abelian $2$-group is in $\G_3$ 
if and only if  it  is of exponent $4,$ and any minimal non-abelian subgroup is isomorphic to $Q_8, \, H_2$ or $H_{32}$
(see Proposition \ref{P 2-Grp}). It is known that there are five minimal non-abelian groups of exponent $4$ (see Corollary \ref{Rc}).  In \cite{Jan07}, Janko described the non-abelian $2$-groups all of whose minimal non-abelian 
subgroups are either of order $8,$ or isomorphic to exactly one of these five groups. 
In particular, if they all are isomorphic to $Q_8,$ then 
$G \cong Q_{2^m} \times E_{2^n},$ where $m \ge 3, \, n \ge 0$ (see \cite[Corollary 2.4]{Jan07}).  
We will make use of this result  when deriving that $Q_8 \times E_{2^n}$ are the only non-abelian 
$2$-groups in $\G_k$ if $k \ge 4$.   

Our goal in this paper is to determine the classes $\G_k$ when $k \ge 4$. Our main result is the following theorem: 

\begin{thm}\label{T main}
Every class $\G_k$ consists of the Cayley integral groups if $k \ge 6$. 
Furthermore, $\G_4 $ and $\G_5$ are equal, and consist of the following groups:  
\begin{enumerate}[(1)]
\item the Cayley integral groups, 
\item the generalized dicyclic groups $Dic(E_3^n \times \Z_6),$ where $n \ge 1$.
\end{enumerate}
\end{thm}

In Section 2 we prove some useful properties of the groups in $\G_k$. 
Theorem \ref{T main} will be derived  in Section 3.

\section{Some properties of the groups in $\mathbf{\G_k}$}

All groups in this paper will be finite. 
Our notation and terminology for finite groups follow \cite{Sco64}.  \medskip

Case (i) in our first lemma is essentially \cite[Lemma 11]{KloS10}, cases (ii) and (iii) can be deduced from 
\cite[Lemma 4.3]{AhmBM13}.

\begin{lem}\label{L basic}
The following hold for every $G \in \G_k$ if $k \ge 2$.
\begin{enumerate}[(i)]
\item For every $x \in G$, the order of $x$  is in $\{1,2,3,4,6\}$. 
\item For every  subgroup $H \le G,$ $H \in \G_k$.
\item For every $N \unlhd G$ such that $|N| \mid k,$ $G/N \in \G_l,$ where $l =  k / |N|$. 
\end{enumerate}
\end{lem}

In contrast to the class of Cayley integral groups, the class $\G_k$ is not closed under forming factor groups for every $k$. 
For example, consider the non-trivial semidirect product $\Z_4 \rtimes \Z_4$. It is easy to see that this group is  in 
$\G_2,$ and that it has a factor group isomorphic to $D_8$. The group $D_8$ is clearly not in $\G_2$. 
Below we prove a weaker property. 

\begin{lem}\label{L section}
Let $G \in \G_k,$ and $N \unlhd G,$ $N$ is abelian and $|N|$ is odd. Then $G/N \in \G_k$.
\end{lem}

Before we prove the lemma, we need to recall  a result in \cite{KovMMM14} about eigenvalues of graphs which admit an 
abelian semiregular automorphism group.

Let $\Gamma$ be a graph, and $H$ be an abelian semiregular 
group of automorphisms  of $\Gamma$ with $m$ orbits on the vertex set.  
Fix $m$ vertices $v_1,\ldots,v_m$ of $\Gamma$  such that no two are from the same $H$-orbit. 
The \emph{symbol} of $\Gamma$ relative to $H$ and the $m$-tuple $(v_1,\ldots,v_m)$ is the 
$m \times m$ array $\S$ of subsets of $H,$ written as $\S = (S_{ij})_{i,j \in \{1,\ldots,m\}},$ where 
\begin{equation}\label{Sij}
S_{ij} = \{ x \in H : v_i  \sim v_j^x \text{ in } \Gamma \}.
\end{equation}
 Here and in what follows, $v_i \sim  v_j^x$ means that the vertices $v_i$ and $v_j^x$ are adjacent in $\Gamma$. 
For an irreducible character $\chi$ of $H,$ let $\chi(\S)$ be the $m \times m$ complex matrix 
defined by 
\begin{equation}\label{chiS}
(\chi(\S))_{ij} = \begin{cases}
\sum_{s \in S_{ij}}\chi(s) & \text{ if } S_{ij} \ne \emptyset \\ 0 & \text{ otherwise, } \end{cases} \; 
i,j \in \{1,\ldots,m\}.
\end{equation}
Note that, since $H$ is abelian, 
the irreducible characters are just the homomorphisms from $H$ to the multiplicative group of complex numbers. 
%The following theorem is \cite[Proposition 3.1]{KovMMM14}: 

\begin{thm}\label{KMMM}
{\rm (Kov\'acs et al.\ \cite{KovMMM14})} With notation as above, the eigenvalues of $\Gamma$ are the 
union of eigenvalues of the matrices $\chi(\S),$ where $\chi$ runs over the set of all irreducible characters of $H$.  
\end{thm}

\noindent{\it Proof of Lemma \ref{L section}.} \ 
Let $\Gamma_{G/N}=\cay(G/N,\R),$  and $\R$ be writen as 
$\R = \{N r : r \in R\}$ for some $R \subset G$. It is clear that $N \cap R = \emptyset,$ in particular, $1 \notin R$. 
Observe that $N r^{-1} = (N r)^{-1} \in \R$. Therefore, if  $(N r)^{-1} \ne N r,$ then
we may assume that both $r$ and $r^{-1}$ are in $R$. 
Let $N r = (N r)^{-1}$. Then $r^2 \in N,$  hence $|\langle N,r \rangle| = 2|N|$. This together with $2 \nmid |N|$ imply that  
$r$ can be chosen to be an involution in $\langle N,r \rangle,$ and so $r^{-1} \in  R$.  
Therefore, we may choose $R$ so that $1 \notin R$ and $R^{-1} = R,$ and thus we have the 
Cayley graph $\Gamma_G=\cay(G,R)$.

Let $m = |G : N|,$ the index of $N$ in $G,$ and $T = \{ t_1,\ldots,t_m\}$ be a complete set of $N$-coset  representatives in $G$ 
such that $R \subseteq T$.  From now on every $x \in N$ will stand also for the permutation of $G$ acting as $g^x = g x, \, g \in G,$ 
and $N$ will stand for the group of all such permutations. Clearly, $N$ is a semiregular group of automorphisms of 
$\Gamma_G$ with $m$ orbits; and as vertices of $\Gamma_G,$  $t_1, \ldots, t_m$  represent all $N$-orbits. 

Let $\bR = (R_{ij})$ be the symbol of $\Gamma_G$ relative to $N$  and the $m$-tuple $(t_1,\ldots,t_m)$.  
By Theorem \ref{KMMM}, the eigenvalues of $\Gamma_G$ are equal to the eigenvalues of $\chi(\bR),$ where $\chi$ runs over the 
set of all irreducible characters of $N$. Therefore, it is sufficient to show that 
the spectrum of $\chi(\bR)$ is equal to the spectrum of $\Gamma_{G/N},$ where 
$\chi$ is the trivial character of $H$ (i.e., $\chi(x)=1$ for every $x \in H$). 
Since $\chi(\bR)_{ij} = |R_{ij}|,$ see \eqref{chiS}, the latter statement follows from the following equivalence:
\begin{equation}\label{iff}
\forall i,j \in \{1,\ldots,m\}: |R_{ij}| \le 1, \text{ and } |R_{ij}|=1 \iff N t_i \sim N t_j  \text{ in } \Gamma_{G/N}. 
\end{equation}

We may write, see \eqref{Sij}, 
\begin{equation}\label{Rij}
R_{ij} = \{ x \in N : t_i  \sim t_j x \text{  in } \Gamma_G \} 
          = \{ x \in N :  t_j x t_i^{-1}  \in R \}| =  N \cap t_j^{-1} R t_i.
\end{equation}
Using that $N \unlhd G,$ this gives that $|R_{ij}|=|N \cap t_j^{-1} R t_i|= |N t_j t_i^{-1} \cap R|$.   
As $R \subseteq T,$ $|N t_j t_i^{-1} \cap R| \le 1$. Furhermore, $|N t_j t_i^{-1} \cap R|=1$ if and only if 
$N t_j t_i^{-1} = N r$ for some $r \in R,$ or equivalently, $N t_j (N t_i)^{-1}  = N  t_j t_i \in \R$ holds in $G/N,$ or 
equivalently, $N t_j$ and $N t_i$ are adjacent in $\Gamma_{G/N}$. This completes the proof of \eqref{iff}.
\hfill $\square$ \medskip

It is well-known that the eigenvalues of a Cayley graph over an arbitrary group $G$ can be computed using the irreducible 
representations of $G$ (see \cite{DiaS81}). In this paper we will rather use Theorem \ref{KMMM}, and hence avoid 
the representation theory of non-abelian groups.  \medskip

In what follows, we write $y^x = x^{-1} y x$ for $x,y \in G,$ and $[x,y]$ will 
denote the \emph{commutator element,} i.e., $[x,y]=x^{-1}y^{-1}x y$.

\begin{lem}\label{L dic}
The group $Dic(E_{3^n} \times \Z_6)$ is in $\G_5$ for every $n \ge 0$. 
\end{lem}

\begin{proof} 
Let $G = Dic(E_{3^n} \times \Z_6)$ and $S \subseteq G$ of size $|S| \le 5$. We have to show that $\cay(G,S)$ is integral. This is clear if 
$|S| \le 2$. 

Let $|S|=3$. Let us write $G = P \rtimes \langle x \rangle,$ where $P \cong E_{3^{n+1}},$ $x$ is of order 
$4,$ and $u^x=u^{-1}$ for every $u \in P$. Notice that $x^2$ is the unique involution of $G$ which is in $Z(G)$. 
Therefore, $x^2 \in S$ and $\langle S \rangle$ is abelian.  This implies that $\cay(G,S)$ is integral. 

Let $|S|=4$. Let $H = \langle P,x^2 \rangle \cong E_{3^{n+1}} \times \Z_2$. 
If $S \subseteq H,$ then $\langle S \rangle \le H,$ and hence $\cay(G,S)$ is integral. 
Thus we may assume that $S$ contains two elements in the form $x u$ and $(x u)^{-1}=x^{-1} u$ for some $u \in H$.  
Let us consider the symbol $\S = (S_{ij})$ of $\cay(G,S)$ relative to $H$ and the pair $(1,x)$. 

Let $S \cap H = \emptyset$. Then $S = \{xu, x^{-1}u, xv, x^{-1}v\},$ where $u,v \in H$ and $u \ne v$.  
According to \eqref{Rij},  the subsets $S_{ij}$ are computed as: 
$S_{11} = H \cap S = \emptyset,$ 
$S_{12} = H \cap x^{-1}S = \{u,x^2u,v,x^2v \},$ 
$S_{21} = H \cap S x  = \{x^2 u^{-1},u^{-1},x^2 v^{-1},v^{-1} \},$ and 
$S_{22} = H \cap x^{-1} S x = \emptyset,$ i.e., 
$$
\S=\begin{pmatrix} \emptyset & \{u,x^2u,v,x^2 v\}  \\   \{x^2u^{-1},u^{-1},x^2 v^{-1},v^{-1}\} & \emptyset  \end{pmatrix}.
$$
For an irreducible character $\chi$ of $H,$ $\chi(\S)$ has eigenvalues  
$\pm (1+\chi(x^2))  |\chi(u)+\chi(v)|$. This is equal to $0$ if $\chi(x^2)=-1$. Otherwise, $\chi(x^2) = 1,$ and 
both $\chi(u)$ and $\chi(v)$ are complex $3^{rd}$ roots of unities,  showing that the eigenvalues are also integers 
in this case. 

Let $S \cap H = \{v,v^{-1}\}$. Then  
$$
\S=\begin{pmatrix} \{v,v^{-1}\} & \{u,x^2u\}  \\   \{u^{-1},x^2u^{-1}\} & \{v,v^{-1}\}  \end{pmatrix}.
$$
The two eigenvalues of $\chi(\S)$ are $\chi(v)+\chi(v^{-1}) \pm (1+\chi(x^2))$. These are also integers, and this completes the 
proof of the case when $|S|=4$. 

Let $|S|=5$. In this case $S$ must contain the unique involution $x^2$. Repeating the above analysis with the set 
$S \setminus \{x^2\},$ one can deduce that $\cay(G,S)$ is always integral. 
\end{proof}

\section{The classes $\mathbf{\G_k, \, k \ge 4}$}

A finite group $G$  is said to be \emph{minimal non-abelian} if all proper 
subgroups of $G$ are abelian. The following result is due to R\'edei \cite{Red47}: 

\begin{thm}\label{R}{\rm (R\'edei \cite{Red47})}   
Let $G$ be a minimal non-abelian $p$-group.  Then $G$ is one of the following groups:
\begin{enumerate}[(i)]
\item $Q_8;$
\item $\big\langle a, b \, | \, a^{p^m}=b^{p^n}=1, a^b=a^{1+p^{m-1}} \big\rangle,$  where $m \ge 2$ (metacyclic);
\item $\big\langle a, b, c \, | \, a^{p^m}=b^{p^n}=c^p=1, [a, b]=c, [c,a]=[c,b]=1  \big\rangle,$  where 
$m+n \ge 3$ if $p=2$ (non-metacyclic).
\end{enumerate}
\end{thm}

\begin{cor}\label{Rc}
The minimal non-abelian groups of exponent at most $4$ are the following:  
\begin{enumerate}[(i)]
\item $Q_8;$
\item  $D_8=\big\langle a,b \, \mid \, a^4=b^2=1, a^b=a^{-1} \big\rangle,$ \\ 
$H_2 = \big\langle a,b \, \mid a^4=b^4=1, a^b=a^{-1}  \big\rangle$ (metacyclic);
\item $H_{16}=\big\langle a, b, c \, | \, a^4=b^2=c^2=1, [a, b]=c, [c,a]=[c,b]=1  \big\rangle ,$ \\ 
$H_{32} =\big\langle a, b, c \, | \, a^4=b^4=c^2=1, [a, b]=c, [c,a]=[c,b]=1  \big\rangle ,$ \\
$H_{27} = \big\langle a,b,c \, \mid \, a^3=b^3=c^3=1, [a,b] = c, [c,a] = [c,b] = 1 \big\rangle$ (non-metacyclic).
\end{enumerate}
\end{cor}

\begin{lem}\label{L p}
Every $p$-group in $\G_k$ is Cayley integral if $k \ge 4$. 
\end{lem}

\begin{proof}
Fix a number $k \ge 4$ and let $G \in \G_k$ be a $p$-group. 
Lemma \ref{L basic}.(i) gives that $p=2$ or $3,$ and the lemma follows at once when $G$ is abelian. 
Assume that $G$ is non-abelian. We have to prove that $G \cong E_{2^n} \times Q_8$ for some $n \ge 0$. 
In view of \cite[Corollary 2.4]{Jan07} (see the introduction), it is sufficient to show that every minimal non-abelian 
subgroup of $G$ is isomorphic to $Q_8$.  

Let $N$ be a minimal non-abelian subgroup of $G$. Note that, $N \in \G_k$ because of Lemma \ref{L basic}.(ii).  
If $p=3,$ then $N \cong H_{27},$ see Corollary \ref{Rc}. 
We exclude this possibility by showing that the graph $\Gamma = \cay(H_{27},\{a,a^{-1},b,b^{-1}\})$ is non-integral. 
Let $H = \langle a,c \rangle \cong E_9,$ and $\S = (S_{ij})$ be the symbol of $\Gamma$ relative to  
$H$ and the triple $(1, b,b^{-1})$. Then compute that   
$$
S_{11} = \{ a,a^{-1} \}, \; S_{22}=\{ac,(ac)^{-1}\}, \; S_{33}= \{ac^{-1},a^{-1}c \},
$$
and $S_{ij}=\{1\}$ if $i \ne j$. Let $\chi$ be the irreducible character of $H$ defined by $\chi(a) = 1$ and 
$\chi(c) = e^{2\pi i/3}$. Then 
$$
\chi(\S) = \begin{pmatrix}  2 & 1 & 1 \\ 1 & -1 & 1 \\  1 & 1 & -1 \end{pmatrix}.
 $$
The eigenvalues of $\chi(\S)$ are $-2$ and $1 \pm \sqrt{3},$ and as these are also eigenvalues of $\Gamma,$ see Theorem \ref{KMMM}, 
$\Gamma$ is indeed non-integral.

Thus $p=2,$ and $N$ is isomorphic to one of the following groups:
$$
Q_8, \; D_8, \; H_2, \; H_{16} \text{ and } H_{32}.
$$

We complete the proof by excluding the last four groups. \medskip

$\bullet$ $D_8:$  Let $D_8 = \langle a,b \mid a^4=b^2=1, bab=a^{-1} \rangle$. It is easy to see that $\cay(H,\{ab,b\})$  
is  isomorphic to an $8$-cycle, which is not integral. 
We actually obtained that $D_8 \notin \G_2$. \medskip

$\bullet$ $H_{2}:$  Notice that $\langle b^2  \rangle \unlhd H_2,$ and $H_2/\langle b^2  \rangle \cong D_8$. 
As $D_8 \notin \G_2,$ this and Lemma \ref{L basic}.(iii) yield that $H_2 \notin \G_k$.   \medskip

$\bullet$ $H_{16}:$ In fact, we show that $H_{16} \notin \G_3$. Consider the graph 
$\cay(H_{16},\{ba,ba^{-1}c,b\})$. Compute its symbol $\S$ relative to  
$H = \langle a,c \rangle \cong \Z_4 \times \Z_2$ and the pair $(1,b):$
$$
\S = \begin{pmatrix}  \emptyset & \{a,a^{-1}c,1\} \\ \{ac,a^{-1},1\} & \emptyset \end{pmatrix} \text{ and } 
\chi(\S) = \begin{pmatrix} 0 & 2i+1 \\ -2i+1 & 0 \end{pmatrix}, 
$$
where $\chi$ is defined by $\chi(a)=i$ (the complex imaginary unit), and $\chi(c)=-1$. 
The eigenvalues of $\chi(\S)$ are $\pm \sqrt{5},$ and so $H_{16} \notin \G_k$ if $k \ge 3$.  \medskip

$\bullet$ $H_{32}:$  Consider the graph $\cay(H_{32},\{ba,b^{-1}a^{-1}c,b,b^{-1}\})$. Compute its symbol $\S$ relative 
to  $H = \langle a,b^2,c \rangle \cong \Z_4 \times E_4$ and the pair $(1,b):$ 
$$
\S = \begin{pmatrix}  \emptyset & \{a,b^2 a^{-1}c,1,b^2\} \\ \{b^2 a c,a^{-1},b^2,1\} & \emptyset \end{pmatrix} \text{ and } 
\chi(\S) = \begin{pmatrix} 0 & 2i+2 \\ -2i+2 & 0 \end{pmatrix}, 
$$
where $\chi$ is defined by $\chi(a)=i$ (the complex imaginary unit), $\chi(b^2)=1$ and $\chi(c)=-1$. 
The eigenvalues of $\chi(\S)$ are $\pm 2\sqrt{2},$ and so $H_{16} \notin \G_k$ if $k \ge 4$. 
\end{proof}

\begin{cor}\label{C nil}
Every nilpotent group in $\G_k$ is Cayley integral if $k \ge 4$. 
\end{cor}  

It is worth to derive the following characterization of non-abelian $2$-groups in $\G_3$.

\begin{prop}\label{P 2-Grp}
Let $G$ be a non-abelian $2$-group of exponent $4$. Then $G \in \G_3$ if and only if every minimal normal subgroup of $G$ 
is isomorphic to $Q_8,\, H_2$ or $H_{32}$. 
\end{prop}

\begin{proof}
As none of $D_8$  and $H_{16}$ is in $\G_3$ (see the above proof), the ``only if'' part follows immediately from 
this and Lemma \ref{L basic}.(ii). 

For the ``if'' part, assume that no subgroup of $G$ is isomorphic to $D_8$ or $H_{16}$. 
It is sufficient to prove that every involution of $G$ is in the center $Z(G)$. It is easy to deduce from this that 
$\langle S \rangle$ is abelian for every inverse-closed subset $S \subset G$ with $1 \notin S$ and $|S| \le 3,$ 
and hence that $G \in \G_3$. 

Assume, towards a contradiction, that $[t,x] \ne 1$ for 
some involution $t$ and element $x$ in $G,$ and let $H=\langle t,x \rangle$. Clearly, $t \notin Z(H)$. 
Since $H \not\cong D_8,$ $x$ must be of order $4$. Also, $x^2 \in Z(H),$ hence $H/\langle x^2 \rangle$ is generated by 
two involutions. It follows that $|H| =8$ or $16$. In the first case, since $H$ is non-abelian, $H \cong Q_8$. 
This contradicts that $t \notin Z(H)$. Therefore, $|H| = 16$. 
If $H$ is minimal non-abelian, then $H \cong H_2$ or $H_{16}$. 
Both cases are impossible, every involution of $H_2$ in is $Z(H_2),$ while the involution $t \notin Z(H),$ and 
$H \not\cong H_{16}$ because of one of the initial assumptions. Thus $H$ contains a non-abelian subgroup of order $8,$ 
say $Q$. Then $Q \cong Q_8,$ and since $t \notin Z(H),$ $t \notin Q,$ and 
$H =   Q \rtimes \langle t \rangle$  is a non-trivial semidirect product.  There is an element $y \in Q$ such that $y^t \ne y$. 
Clearly, $y$ is of order $4$. If $y^t = y^{-1},$ then $\langle y,t \rangle \cong D_8,$ a contradiction. If $y^t \ne y^{-1},$ 
then putting $z = y y^t,$ we find that $z$ is of order $4,$ and $z^t = (y y^t)^t = y^t y = z^{-1}$. Thus 
$\langle z,  t \rangle \cong D_8,$ a contradiction. This completes the proof of the proposition. 
\end{proof}

Now, we return to the classes $\G_k, \, k \ge 4$.

\begin{lem}\label{L normal}
Suppose that $G \in \G_k, \, k \ge 4,$ and $3 \mid |G|$. Then $G$ has a normal Sylow $3$-subgroup. 
\end{lem}

\begin{proof}
We proceed by induction on the order of $G$. There is nothing to prove if $G$ is a $3$-group, hence 
we may assume that $2$ and $3$ are the prime divisors of $|G|,$ see Lemma \ref{L basic}.(i).
Burnside's ``pq'' Theorem gives that $G$ is solvable. 
Let $K$ be a minimal normal subgroup of $G$. It is well-known that $K$ is elementray abelian, and hence $K < G$. 
Let $M$ be a maximal normal subgroup of $G$ which contains $K$. Then $G/M$ is simple and solvable 
(see \cite[2.5.2 and 2.6.1]{Sco64}), which imply that $G/M$ is of prime order (see \cite[Exerxice 2.6.6]{Sco64}).
Therefore, $|G : M|=2$ or $3$. 
Also note that, $M \in \G_k$ because of Lemma \ref{L basic}.(ii).
 
If $|G : M|=2,$ then the induction hypothesis gives that $M$ has a normal Sylow $3$-subgroup,  
which is clearly also a normal Sylow $3$-subgroup of $G$. 

Let $|G : M|=3,$ and suppose that $3 \mid |M|$.  Then the 
the induction hypothesis gives that $M$ has a normal Sylow $3$-subgroup, say $N$. 
By Lemma \ref{L section}, $G/N \in \G_k,$ hence $G/N$  has a normal Sylow $3$-subgroup, say $L$ ($L \cong \Z_3$). 
Then the pre-image $\eta^{-1}(L),$ where $\eta : G \to G/N$ is the natural projection, is a normal Sylow $3$-subgroup 
in $G$. 

We are left with that case that $9 \nmid |G|,$ and $G$ has a normal sylow $2$-subgroup. 
Thus $G = P \rtimes \langle x \rangle,$ where $P$ is a $2$-group and $x$ is of order $3$. 
We complete the proof by showing that $x$ centralizes $P,$ and thus $G$ is abelian. 

Assume, towards a contradiction, that 
$[u,x] \ne 1$ for some $u \in P$. Let $U=\langle u,u^x,u^{x^2} \rangle$ and $V=\langle u, x \rangle$.
Clearly, $U \unlhd V$ and $V \cong U \rtimes \Z_3$. 

Suppose for the moment that $u$ is of order $2$.  
The group $P$ is Cayley integral, see Lemma \ref{L p}, in particular, all involutions of 
$P$ are in the center $Z(P)$. This implies that $U \cong E_4$ or $E_8$. 
If  $U \cong E_4,$ then $V = E_4 \rtimes \Z_3 \cong A_4$. If $U \cong E_8,$ then consider the group 
$\langle u u^x, u u^{x^2}, x \rangle$. As this is isomorphic to $A_4,$ we 
see that in either case, $V$ contains a subgroup isomorphic to $A_4$.
We show next that this is impossible by proving that $A_4 \not\in \G_k$.  
Write $A_4=\langle (1,3)(2,4), (1,2,3) \rangle,$ and let $\Gamma=\cay(A_4,\{a,b,c,c^{-1}\}),$ 
where $a=(1,2)(3,4),$ $b=(1,3)(2,4)$ and $c=(1,2,3)$. Let $\S$ be the symbol of $\Gamma$ relative to  
$H = \langle a,b \rangle$ and the triple $(1,c,c^{-1}),$ and let $\chi$ be the irreducible character of $H$ 
defined by $\chi(a)=1$ and $\chi(b)=-1$. Then 
$$
\S=\begin{pmatrix}  \{a,b\} & \{1\} & \{1\} \\ \{1\} & \{ab,a\} & \{1\} \\ \{1\} & \{1\} & \{b,ab\} \end{pmatrix} \text{ and } 
\chi(\S) = \begin{pmatrix}  0 & 1 & 1 \\  1 &  0 & 1 \\ 1 & 1 & -2 \end{pmatrix}.
$$
The eigenvalues of $\chi(\S)$ are $-1$ and $\frac{1}{2}(-1 \pm \sqrt{17}),$ hence $A_4 \notin \G_k$. 
We conclude that $x$ centralizes all involutions of $P$. 

Let $u$ be of order $4,$ and $W=\langle u \rangle$. Since $[u,x] \ne 1,$ $W^x \ne W,$ and thus    
$W, W^x$ and $W^{x^2}$ are three distinct subgroups of order $4$ contained in $U$. Since $x$ centralizes all involutions 
of $P,$ it follows that  
\begin{equation}\label{cap}
W \cap W^x \cap W^{x^2} = \langle u^2 \rangle.
\end{equation}  

Suppose that $P$ is abelian. Then $P \cong E_{2^m} \times \Z_4^n$ for some $m \ge 0, n \ge 1$. 
Using that $U/\langle u^2 \rangle$ is elementary abelian of order $4$ or $8,$ we deduce that 
$U \cong \Z_2 \times \Z_4$ or $U \cong E_4 \times \Z_4$. The first case cannot occur, $U$ has three distinct 
subgroups of order $4$. In the second case there are four subgroups in $U$ of order $4$ containing $u^2,$ 
and thus must be one normalized by $x$. This, however, gives rise to an element in $G$ of order $12,$ and 
this is impossible.  

Let $P$ be non-abelian.  Then $P \cong E_{2^m} \times Q_8.$ All subgroups of $P$ of order $4$ intersect at 
the same subgroup, the Frattini subgroup $\Phi(P)$. Thus \eqref{cap} gives that $\Phi(P) = \langle u^2 \rangle,$ 
$U/\langle u^2 \rangle \cong E_4$ or $E_8,$ and $U \cong Q_8$ or $U \cong \Z_2 \times Q_8$ respectively. 
If $U \cong Q_8,$ then $V \cong Q_8 \rtimes \Z_3$. 
Let $U \cong Q_8 \times \Z_2$. Then $U$ contains exactly four subgroups isomorphic to 
$Q_8$. Thus one of them must be normalized by $x,$ but not centralized, hence we see that 
$V$ always contains a subgroup isomorphic to $Q_8 \rtimes \Z_3$. 
We finish the proof by showing that $Q_8 \rtimes \Z_3 \not\in \G_k$.  
Write $Q_8=\{\pm 1,\pm i,\pm j,\pm k\}$ (the usual quaternion group) and 
$Q_8 \rtimes \Z_3 = Q_8 \rtimes \langle \sigma \rangle,$ where 
$\sigma^3=1,$ $[-1,\sigma]=1,$ 
$i^\sigma=j, j^\sigma=k,$ and $ k^\sigma=i$.  Let $\Gamma = \cay(Q_8 \rtimes \Z_3, \{i,-i,\sigma,\sigma^{-1}\}),$ 
and $\S$ be the symbol of $\Gamma$ relative to  
$H = \langle -1,\sigma \rangle$ and the quadruple $(1,i,j,k)$. Let $\chi$ be the trivial character of $H$. Then 
$$
\S= \begin{pmatrix}  
\{\sigma,\sigma^{-1}\} & \{1,-1\} & \emptyset & \emptyset \\
\{1,-1\} & \emptyset & \{\sigma^{-1}\} & \{\sigma\} \\
\emptyset & \{\sigma\} & \emptyset & \{1,-1,\sigma^{-1}\} \\
\emptyset & \{\sigma^{-1}\} & \{1,-1,\sigma\} & \emptyset 
       \end{pmatrix} \text{ and }
\chi(\S) = \begin{pmatrix}  2 & 2 & 0 & 0 \\  2 & 0 & 1 & 1 \\ 0 & 1 & 0 & 3 \\ 0 & 1 & 3 & 0 \end{pmatrix}.
$$
The eigenvalues of $\chi(\S)$ are $4,$ $-3,$ and $\frac{1}{2}(1 \pm \sqrt{17}),$
hence $Q_8 \rtimes \Z_3 \notin \G_k,$ as claimed. 
\end{proof}

\begin{lem}\label{L non-nil}
Suppose that $G \in \G_k, \, k \ge 4,$ and $G$ is not nilpotent. Then $G \cong D_6$ or 
$Dic(E_3^n \times \Z_6)$ for some $n \ge 0$.
\end{lem}

\begin{proof}
By Lemma \ref{L normal}, $G$ contains a normal Sylow $3$-subgroup, say $P$. 
Since $G$ is non-abelian, there is an element $x$ of order $2$ or $4$ such that $x \notin C_G(P)$.

We consider first the case when $x$ is of order $2$. Suppose that $w^x \notin  \langle w \rangle$ for some $w \in P$. 
Let  $u=w^x w,$ $v=w^x w^{-1}$ and $U = \langle u, v \rangle$. 
Then $U \cong E_9,$ $u^x= u$ and $v^x = v^{-1}$. Thus $V = \langle U, x \rangle \cong D_6 \times \Z_3$. 
We exclude this possibility by showing that $V \notin \G_k$. 
Let $\Gamma = \cay(V,\{ x u,x u^{-1},x v \}),$  $\S$ be the symbol of $\Gamma$ relative to 
$U$ and the pair $(1,x),$ and let $\chi$ be the irreducible character of $U$ defined by $\chi(u)=1$ and $\chi(v)=\xi=e^{2i \pi/3}$. 
Then 
$$
\S=\begin{pmatrix}  \emptyset & \{u,u^{-1},v\} \\ \{u,u^{-1},v^{-1}\} & \emptyset \end{pmatrix} \text{ and } 
\chi(\S) = \begin{pmatrix}  0 & 2+\xi \\  2+\xi^{-1} & 0 \end{pmatrix}.
$$
The eigenvalues of $\chi(\S)$ are $\pm \sqrt{3},$ hence $V \notin \G_k$. 
We are left with the case that $x$ inverts all elements of $P$. Assume that $|P| > 3$. Then let 
$U= \langle u,v \rangle \cong E_9,$ and  $V = \langle U,x \rangle \cong E_9 \rtimes \Z_2$. 
Copying above argument for the graph $\Gamma = \cay(V,\{ xu,xu^{-1},xv\}),$ we find again 
that $\pm \sqrt{3}$ are eigenvalues of $\Gamma,$ a contradiction. Therefore, $|P|=3$. 
 Let  $N = C_G(P),$ the centralizer of $P$ in $G$.   
Notice that, $N$ is an abelian normal subgroup of $G,$ and by the N/C Theorem (see \cite[Theorem 3.2.3]{Sco64}), 
$G/N$ is isomorphic to a subgroup of Aut$(P) \cong \Z_2$. Since $G$ is non-abelian, $G \ne N,$ and thus $|G : N|=2$. 
If $N \ne P,$ then take an involution $y \in N$. Then $y \in Z(G),$ hence 
$\langle x,y,P \rangle = \langle x,u \rangle \times \langle y \rangle \cong D_6 \times \Z_2
\cong D_{12},$ a contradiction. Therefore, $N=P,$  and $G \cong D_6$.  

To sum up, we may assume that $G \not\cong D_6$ and  all involutions of $G$ are in $C_G(P)$. 
Let $x \in G \setminus C_G(P),$ and let $Q$ be a Sylow $2$-subgroup such that $x \in Q$.  
Suppose that  $u^x \notin \langle u \rangle$ for some $u \in P$. 
Then using that $[x^2,u]=1,$ we can prove, as above, that $v^x = v$ for some $v \in P,$ hence 
$v x$ is of order $12,$ a contradiction. Therefore, $u^x = u^{-1}$ for every $u \in P$.  
Let $y$ be an involution of $Q$ such that $y \ne x^2,$ and let $V =  \langle x,y,u \rangle,$ where $u  \in P, \, u \ne 1$. 
Then $y \in Z(G),$ hence $V = \langle x,u \rangle \times \langle y \rangle \cong Dic_{12} \times \Z_2$. 
Then $\langle x^2 \rangle \unlhd V,$ and $V/\langle x^2 \rangle \cong D_6 \times \Z_2 \cong D_{12}$.  
This and Lemma \ref{L basic}.(iii) yield that $V \notin \G_k,$ a contradiction.  
Thus $Q$ has a unique involution, and hence $Q \cong \Z_4$ or $Q_8$ (recall that $Q$ is Cayley integral). 
Let $Q \cong Q_8$ and $K = \langle u, Q \rangle$ for some $u \in P, \, u \ne 1$. 
Then $\langle u \rangle \unlhd K,$ 
$|K| = 24,$ and the centralizer $C_K(u)$ is of order at least $12$. 
This shows that  $C_K(u)$ contains an element of order $4,$ and so $K$ contains element of order $12,$ a contradiction.  
Therefore, $Q \cong \Z_4,$ and $G \cong Dic(E_{3^n} \times \Z_6),$ where $n \ge 0$. This completes the proof of the lemma. 
\end{proof}

Everything is prepared to derive the main theorem. \medskip

\noindent{\it Proof of Theorem \ref{T main}.} Fix a number $k \ge 4,$ and let $G \in \G_k$ be a group which is not Cayley integral.  
By Corollary \ref{C nil}, $G$ is not nilpotent, hence by Lemma \ref{L non-nil}, $G \cong Dic(E_{3^n} \times \Z_6)$ for some 
$n \ge 1$ (here we use that  $Dic(\Z_6) \cong Dic_{12},$ which is Cayley integral). 
By Lemma \ref{L dic}, these groups are also in $\G_4$ and $\G_5,$ and this settles the 
the second part of the theorem. 

It remains to prove that $Dic(E_{3^n} \times \Z_6) \notin \G_k$ if $n \ge 1$ and $k  \ge 6$.  
Observe that, all these groups contain a subgroup isomorphic to $Dic(\Z_3 \times \Z_6)$.
Therefore, it is sufficient to show that $Dic(\Z_3 \times \Z_6) \notin \G_k$ if $k \ge 6$ (see  Lemma \ref{L basic}.(ii)). 
Write $Dic(\Z_3 \times \Z_6) = E \rtimes \langle x \rangle,$ where $E \cong E_9,$ $x$ is of order $4,$ and $x$ inverts every 
element in $E$. Then $\langle x^2 \rangle$ is normal in $E \rtimes \langle x \rangle,$ and 
$(E \rtimes \langle x \rangle)/\langle x^2 \rangle \cong E_9 \rtimes \Z_2$. However, the latter group is not in $\G_3$ 
(see the proof of Lemma \ref{L normal}). This and Lemma \ref{L basic}.(iii) yield that $E \rtimes \langle x \rangle \notin \G_k$ 
if $k \ge 6$. This completes the proof of the theorem. 
\hfill $\square$


\begin{thebibliography}{9}
%\bibitem{AbdJ13} 
%A. Abdollahi, M. Jazaeri, 
%On groups admitting no integral Cayley graphs besides complete multipartite graphs, 
%{\it Appl. Anal. Discrete  Math.} {\bf 7} (2013), 119--128.  

\bibitem{AbdV09} 
A. Abdollahi, E. Vatandoost, 
Which Cayley graphs are integral, 
{\it Electronic J. Combin.}   {\bf 16} (2009), R122, 1--17.  

\bibitem{AbdJ14} 
A. Abdollahi, M. Jazaeri, 
Groups all of whose undirected Cayley graphs are integral
{\it Europ. J. Combin.} {\bf 38} (2014), 102--109.

\bibitem{AhmABS09}
O. Ahmadi, N. Alon, L. F.  Blake,  I. E. Shparlinski, 
Graphs with integral spectrum, 
{\it Linear Algebra Appl.} {\bf 430} (2009), 547–-552.

\bibitem{AhmBM13}
A. Ahmady, J. P. Bell, B. Mohar,
Integral Cayley graphs and groups, 
preprint arXiv:1209.5126v1 [math.CO] 2013. 

%\bibitem{AhmDKM14}
%A. Ahmady, M. DeVos, R. Krakovski, B. Mohar, 
%Integral Cayley multigraphs over Abelian and Hamiltonian groups, 
%{\it Electronic J. Combin.}, to appaer.

\bibitem{BalCRSS03}
K. Bali\'nska, D. Cvetkovi\'c, Z. Radosavljevi\'c, S. Simi\'c, D. Stevanovi\'c, 
A survey on integral graphs, 
{\it Univ. Beogr. Publ. Elektrotehn. Fak. Ser. Mat} {\bf 13} (2003), 42–-65.

%\bibitem{Ber08}
%Y. Berkovich,
%{\it Groups of prime power order -  Volume 1,}
%Valter de Gruyter GmbH \& Co., Berlin New-York 2008.

\bibitem{DiaS81}
P. Diaconis, M. Shahshahani, 
Generating a random permutation with random transpositions,
{\it Z. Wahsch. Verw. Gebiete} {\bf 57} (1981), 159--179.

\bibitem{HarS74}
F.  Harary, A. J. Schwenk, 
Which graphs have integral spectra? in  ``Graphs and Combinatorics
(Proc. Capital Conf., George Washington Univ., Washington, D.C., 1973)'', Lecture
Notes in Mathematics {\bf 406.} Springer, Berlin, 1974, 45–-51.

\bibitem{Jan07}
Z. Janko,
On finite nonabelian $2$-groups all of whose minimal nonabelian subgroups are of exponent $4,$
{\it J. Algebra.} {\bf 315} (2007), 801--808.

\bibitem{KloS10} 
W. Klotz, T. Sander, 
Integral Cayley graphs over abelian groups,
{\it Electronic J. Combin.} {\bf 17} (2010), \#R81.

\bibitem{KovMMM14}
I. Kov\'{a}cs, D. Maru\v{s}i\v{c}, A. Malni\v{c}, \v{S}. Miklavi\v{c}, 
Transitive group actions: (im)primitivity and semiregular subgroups, submitted preprint.
%preprint arXiv:math/0701686v1 [math.GR] 2007.

\bibitem{Red47}
L. R\'{e}dei, 
Das schiefe Product in der Gruppentheorie,
{\it Comment. Math. Helv.} {\bf 20} (1947), 225–-267.

\bibitem{Sco64} 
W. R. Scott,  
Group theory, 
Prentice-Hall, New Jersey 1964.
\end{thebibliography}
\end{document}